\documentclass[a4paper,twoside,12pt]{article}

\usepackage{amssymb,amsmath,amsthm,latexsym}
\usepackage{amsfonts}
\usepackage{graphicx,caption,subcaption}
\usepackage[pdftex,bookmarks,colorlinks=false]{hyperref}
\usepackage[hmargin=1.2in,vmargin=1.2in]{geometry}
\usepackage[mathscr]{euscript}
\usepackage[auth-lg,affil-sl]{authblk}

\usepackage[auth-lg,affil-sl]{authblk}
\allowdisplaybreaks

\usepackage{float}
\usepackage{calc,tikz}
\usetikzlibrary{arrows,decorations.markings}
\tikzstyle{vertex}=[circle, draw, inner sep=1pt, minimum size=8pt]
\newcommand{\vertex}{\node[vertex]}

\usepackage{multirow}
\usepackage{hhline}
\usepackage{booktabs}
\usepackage{longtable}

\newcommand{\noi}{\noindent}
\newcommand{\N}{\mathbb{N}}
\newcommand{\C}{\mathcal{C}}

\newtheorem{theorem}{Theorem}[section]
\newtheorem{definition}[theorem]{Definition}
\newtheorem{lemma}[theorem]{Lemma}
\newtheorem{proposition}[theorem]{Proposition}
\newtheorem{corollary}[theorem]{Corollary}

\newtheorem{conjecture}{Conjecture}

\newtheorem{desc}{Colouring Description}

\title{\textbf{\sc On the Rainbow Neighbourhood Number of Set-Graphs}}

\author{Johan Kok$^\ast$, Sudev Naduvath$^\dagger$}
\affil{\small Centre for Studies in Discrete Mathematics\\ Vidya Academy of Science \& Technology \\Thalakkottukara, Thrissur, Kerala, India.\\ $^\ast${\tt kokkiek2@tshwane.gov.za}\\$^\dagger${\tt sudevnk@gmail.com}}

\date{}

\begin{document}
\maketitle

\begin{abstract}
\noi In this paper, we present results for the rainbow neighbourhood numbers of set-graphs. It is also shown that set-graphs are perfect graphs. The intuitive colouring dilemma in respect of the rainbow neighbourhood convention is clarified as well. Finally, the new notion of the maximax independence, maximum proper colouring of a graph and a new graph parameter called the $i$-max number of $G$ are introduced as a new research direction.
\end{abstract}

\noi\textbf{Keywords:}  Rainbow neighbourhood, rainbow neighbourhood number, set-graph, maximax independence, maximum proper colouring, $i$-max number.
\vspace{0.25cm}

\noi \textbf{Mathematics Subject Classification 2010:} 05C15, 05C38, 05C75, 05C85.


\section{Introduction}

For general notation and concepts in graphs and digraphs see \cite{BM1,FH,DBW}. Unless stated otherwise, all graphs we consider in this paper are simple, connected and finite graphs.

For a set of distinct colours $\mathcal{C}= \{c_1,c_2,c_3,\dots,c_\ell\}$ a vertex colouring of a graph $G$ is an assignment $\varphi:V(G) \mapsto \mathcal{C}$ . A vertex colouring is said to be a \textit{proper vertex colouring} of a graph $G$ if no two distinct adjacent vertices have the same colour. The cardinality of a minimum set of colours in a proper vertex colouring of $G$ is called the \textit{chromatic number} of $G$ and is denoted $\chi(G)$. We call such a colouring a \textit{chromatic colouring} of $G$. 

A \textit{minimum parameter colouring} of a graph $G$ is a proper colouring of $G$ which consists of the colours $c_i;\ 1\le i\le \ell$, with minimum subscripts $i$. Unless stated otherwise, we consider minimum parameter colouring throughout this paper. The set of vertices of $G$ having the colour $c_i$ is said to be the \textit{colour class} of $c_i$ in $G$ and is denoted by $\C_i$. The cardinality of the colour class $\C_i$ is said to be the weight of the colour $c_i$, denoted by $\theta(c_i)$. Note that $\sum\limits_{i=1}^{\ell}\theta(c_i)=\nu(G)$. 

Unless mentioned otherwise,  we colour the vertices of a graph $G$ in such a way that $\C_1=I_1$, the maximal independent set in $G$, $\C_2=I_2$, the maximal independent set in $G_1=G-\C_1$ and proceed like this until all vertices are coloured. This convention is called \textit{rainbow neighbourhood convention}(see \cite{KNJ,KN1}).

Unless mentioned otherwise, we shall colour a graph in accordance with the rainbow neighbourhood convention. 

Note that the closed neighbourhood $N[v]$ of a vertex $v \in V(G)$ which contains at least one coloured vertex of each colour in the chromatic colouring, is called a rainbow neighbourhood (see \cite{KNJ}). The number of vertices in $G$ which yield rainbow neighbourhoods,  denoted by $r_\chi(G)$, is called the \textit{rainbow neighbourhood number} of $G$. 

In \cite{KNB}, the bounds on $r_\chi(G)$ corresponding to of minimum proper colouring was defined as, $r^-_\chi(G)$ and $r^+_\chi(G)$ respectively denote the minimum value and maximum value of $r_\chi(G)$ over all permissible colour allocations. If we relax connectedness, it follows that the null graph $\mathfrak{N}_n$ of order $n\geq 1$ has $r^-(\mathfrak{N}_n)=r^+(\mathfrak{N}_n)=n$. For more results in this area, refer to \cite{KN2}. For bipartite graphs and complete graphs, $K_n$ it follows that, $r^-(G)=r^+(G)=n$ and $r^-(K_n)=r^+(K_n)=n$.

We observe that if it is possible to permit a chromatic colouring of any graph $G$ of order $n$ such that the star subgraph obtained from vertex $v$ as center and its open neighbourhood $N(v)$ the pendant vertices, has at least one coloured vertex from each colour for all $v \in V(G)$ then $r_\chi(G) = n$. Certainly to examine this property for any given graph is complex. More interesting results are presented in \cite{SNKK,NSKK}.

\begin{lemma}
{\rm \cite{KCSS}} For any graph $G$, the graph $G'=K_1 + G$ has $r_\chi(G')=1+r_\chi(G)$.
\end{lemma}

\section{Rainbow Neighbourhood Number of Set-graphs}

\noi This section begins with an important lemma from \cite{KNB}.

\begin{lemma}\label{Lem-2.1} 
{\rm \cite{KNB}} If a vertex $v \in V(G)$ yields a rainbow neighbourhood in $G$ then, $d_G(v) \geq \chi(G)-1$.
\end{lemma}

\noi Lemma \ref{Lem-2.1} motivated the next probability corollary.

\begin{corollary}\label{Cor-2.2}
{\rm \cite{KNB}} A vertex $v\in V(G)$ possibly yields a rainbow neighbourhood in $G$ if and only if, $d_G(v) \geq \chi(G)-1$.
\end{corollary}

\noi The notion of a set-graph was introduced in \cite{KCSS}  as explained below.

\begin{definition}{\rm 
\cite{KCSS}  Let $A^{(n)} = \{a_1,a_2,a_3,\dots , a_n\}$, $ n \in \N$ be a non-empty set and the $i$-th $s$-element subset of $A^{(n)}$ be denoted by $A^{(n)}_{s,i}$. Now, consider $\mathcal S = \{A^{(n)}_{s,i}: A^{(n)}_{s,i} \subseteq A^{(n)}, A^{(n)}_{s,i} \neq \emptyset \}$. The \textit{set-graph} corresponding to set $A^{(n)}$, denoted $G_{A^{(n)}}$, is defined to be the graph with $V(G_{A^{(n)}}) = \{v_{s,i}: A^{(n)}_{s,i} \in \mathcal S\}$ and $E(G_{A^{(n)}}) = \{v_{s,i}v_{t,j}: A^{(n)}_{s,i} \cap A^{(n)}_{t,j} \neq \emptyset\}$, where $s\neq t$ or $i\neq j$.
}\end{definition}

Figure \ref{fig:Fig-1SetGraph} depicts the set-graph $G_{A^{(3)}}$.

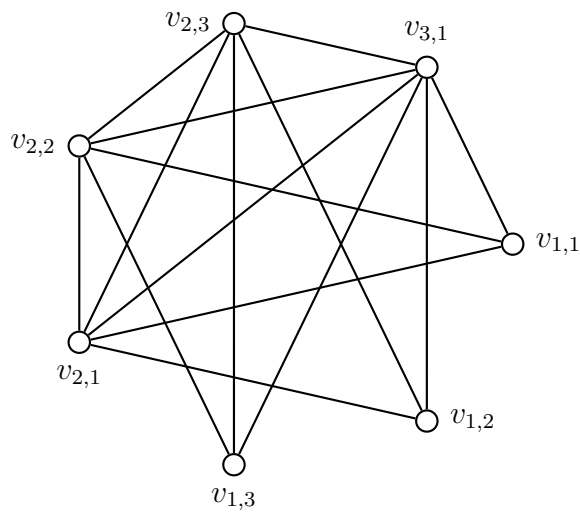
\begin{figure}[h!]
\centering
\begin{tikzpicture}[auto,node distance=1.75cm,
thick,main node/.style={circle,draw,font=\sffamily\Large\bfseries}]
\vertex (v1) at (0:3) [label=right:$v_{1,1}$]{};
\vertex (v2) at (308.57:3) [label=right:$v_{1,2}$]{};
\vertex (v3) at (257.14:3) [label=below:$v_{1,3}$]{};
\vertex (v4) at (205.71:3) [label=below:$v_{2,1}$]{};
\vertex (v5) at (154.28:3) [label=left:$v_{2,2}$]{};
\vertex (v6) at (102.85:3) [label=left:$v_{2,3}$]{};
\vertex (v7) at (51.42:3) [label=above:$v_{3,1}$]{};
\path 
(v1) edge (v4)
(v1) edge (v5)
(v1) edge (v7)
(v2) edge (v4)
(v2) edge (v6)
(v2) edge (v7)
(v3) edge (v5)
(v3) edge (v6)
(v3) edge (v7)
(v4) edge (v5)
(v4) edge (v6)
(v4) edge (v7)
(v5) edge (v6)
(v5) edge (v7)
(v6) edge (v7)
;
\end{tikzpicture}
\caption{\small The set-graph $G_{A^{(3)}}$.}\label{fig:Fig-1SetGraph}
\end{figure}

The largest complete graph in the given set-graph $G_{A^{(n)}}$, $n \geq 2$  is $K_{2^{n-1}}$ and the number of such largest complete graphs in the given set-graph $G_{A^{(n)}}$, $n \geq2$ is provided in the following proposition.

\begin{proposition}\label{Prop-2.1}
{\rm \cite{KCSS}} The set-graph $G_{A^{(n)}}, n\geq 1$ has exactly $2^{n-1}$ largest complete graphs $K_{2^{n-1}}$.
\end{proposition}

\begin{theorem} \label{Thm-2.2}
{\rm \cite{KCSS}} The chromatic number of a set-graph $G_{A^{(n)}}$, $n \geq 1$  is $\chi(G_{A^{(n)}}) = 2^{n-1}$.
\end{theorem}

\noi The next important theorem is found to hold.

\begin{theorem}\label{Thm-2.3}
A set-graph $G_{A^{(n)}}$, $n \geq 1$  is a perfect graph.
\end{theorem}
\begin{proof}
From Proposition \ref{Prop-2.1} it follows that $\omega(G_{A^{(n)}}) =2^{n-1}~and~2^{n-1} = \chi(G_{A^{(n)}})$, follows from Theorem \ref{Thm-2.2}. Hence, a set-graph is weakly perfect. Since a set-graph has $\omega(G_{A^{(n)}}) =2^{n-1}$ and exactly $2^n-1$ vertices it has a unique maximum independent set $X$ of $\alpha(G_{A^{(n)}}) = 2^{n-1}-1$ vertices. Furthermore, $\langle X\rangle$ is a null graph hence, any subgraph thereof is perfect.

Also, since $\omega(G_{A^{(n)}}) =2^{n-1}~and~2^{n-1} = \chi(G_{A^{(n)}})$, each vertex $v \in \langle X\rangle$ is in some induced maximum clique, therefore each vertex in $V(G_{A^{(n)}})$ is in some induced maximum clique. It then follows that $\omega(H) = \chi(H)$, $\forall H\subseteq G_{A^{(n)}}$, $n \geq 1$. Hence, the result. 
\end{proof}

\begin{conjecture}\label{Conj}
Any weakly perfect graph $G$ is a perfect graph if each vertex $v \in V(G)$ is in some maximum clique of $G$.
\end{conjecture}

\noi Now for the main result of this section.

\begin{theorem}
For a set-graph $G_{A^{(n)}}$, $n \geq 1$ the maximum and minimum rainbow neighbourhood number are $r^-_\chi(G_{A^{(n)}}) = r^+_\chi(G_{A^{(n)}}) = 2^n-1$.
\end{theorem}
\begin{proof}
The result is a direct consequence of the result that $\nu(G_{A^{(n)}}) = 2^n - 1$, Proposition \ref{Prop-2.1} and Theorem \ref{Thm-2.2}.
\end{proof}

It is observed that $\chi(G) \leq \delta(G)+1$ if and only if $r^-_\chi(G)=n$. For set-graphs and paths, equality of this minimum upper bound holds whilst for cycles $C_n$, where $n$ is even, the inequality holds.

\section{Maximax Independence, Maximum Proper Colouring of a Graph}

From \cite{KCSS} it is known that the independence number of a set-graph is, $\alpha(G_{A^{(n)}}) = n$. This observation challenged the clear understanding of the rainbow neighbourhood convention since it was discovered that $\theta(c_1) \neq n$ for set-graphs, $G_{A^{(n)}}$, $n\geq 3$. Another example which highlights this intuitive dilemma is discussed next.
\vspace{0.25cm}

\textbf{Example 1.} Consider a thorn(y) complete graph $K^\star_n$, $n \geq 3$ with the vertices of $K_n$ labelled $v_1,v_2,v_3,\dots,v_n$. Without loss of generality assume that $t_i \geq 1$ pendant vertices (thorns) are attached to each vertex $v_i$ such that, $t_1 \leq t_2 \leq \cdots \leq t_n$. Label these vertices $u_{i,j}$, $1\leq j \leq t_i$. Clearly, colouring $c(v_1) = c_1$ and $c(u_{i,j}) = c_1$, $2\leq i \leq n$ as well as $c(u_{1,j}) = c_2$, $1 \leq j \leq t_1$ permits the rainbow neighbourhood convention in accordance with the minimum proper colouring. Hence, $|\mathcal{C}| = \chi(K^\star_n)$. Therefore, $\theta(c_1) \neq \sum\limits_{i = 1}^{n}t_i$ as may be intuitively expected. In fact, if $t_1 < t_2 < \cdots < t_n$, the minimum proper $n$-colouring is unique.

The observations above motivate the new notion of maximax-independence, maximum proper colouring of a graph. It is described as follows.

\begin{desc}\label{Col-3.1}{\rm 
Let $G$ have the distinct maximum independent sets, $X_i$, $1\leq i \leq k$. Consider the graph, $G' = G-X_i$ for some $X_i$ such that $\alpha(G') = min\{\alpha(G-X_j):1\leq j \leq k\}$. Call the aforesaid procedure the first iteration and let $X_1 = X_i$. Colour all the vertices in $X_1$ the colour $c_1$. Repeat the procedure iteratively with $G'$ to obtain $G''$,~then~ $G'''~then, \cdots$,~then~empty graph. Colour all the vertices in the maximum independent set obtained from the $j^{th}$-iteration the colour, $c_j$. This colouring is called a \textit{maximax independence, maximum proper colouring} of a graph.
}\end{desc}

From Figure \ref{fig:Fig-1SetGraph}, it follows that the colouring $c(v_{1,1}) = c(v_{1,2})= c(v_{1,3}) = c_1$ and without loss of generality say, $c(v_{2,1}) =c_2$, $c(v_{2,2}) = c_3$, $c(v_{2,3}) = c_4$, $c(v_{3,1}) = c_5$ permits a maximax-independence, maximum proper colouring of the set-graph $G_{A^{(3)}}$. For a graph $G$ this invariant is denoted, $\chi^{i-max}(G)$.

The empty-sun graph denoted $S^{\bigodot}_n = C_n\bigodot K_1$ is obtained by attaching an additional vertex $u_i$ with edges $u_iv_i$ and $u_iv_{i+1}$ for each edge $v_iv_{i+1} \in E(C_n)$, $1\leq i \leq n-1$ and similarly vertex $u_n$ with corresponding edges, $u_nv_n$, $u_nv_1$ in respect of edge $v_nv_1$.

\begin{proposition}\label{Prop-3.1}
\begin{enumerate}\itemsep0mm
\item[(a)] For a set-graph, $G_{A^{(n)}}$, $n\geq 3$ we have: $\chi^{i-max}(G_{A^{(n)}}) = \chi(G_{A^{(n)}}) + 1$.
\item[(b)] For a path $P_n$, and if and only if $n \geq 4$ and even, we have $\chi^{i-max}(P_n) = \chi(P_n) + 1 = 3$.
\item[(c)]  For a path $P_n$, $n \geq 4$ and odd, we have $\chi^{i-max}(P_n) = \chi(P_n) = 2$.
\item[(d)] For a cycle $C_n$, $n \geq 3$, we have $\chi^{i-max}(C_n) = \chi(C_n) = 2~or~3$, ($n$ is even or odd).
\item[(e)] For a sunlet graph, $S_n$, $n \geq 3$, we have $\chi^{i-max}(S_n) = \chi(S_n) + 1 = 3~or~4$, ($n$ is even or odd). 
\item[(f)] For a empty-sun graph, $S^{\bigodot}_n$, $n \geq 3$, we have $\chi^{i-max}(S^{\bigodot}_n) = \chi(S^{\bigodot}_n)+ 1 = 4$, ($n$ is odd). 
\item[(g)] For a empty-sun graph, $S^{\bigodot}_n$, $n \geq 4$, we have $\chi^{i-max}(S^{\bigodot}_n) = \chi(S^{\bigodot}_n)= 3$, ($n$ is even). 
\item[(h)] For a thorn complete graph $K^\star_n$, $n\geq 3$ and thorns $t_i \geq 1$, we have $\chi^{i-max}(K^\star_n) = n+1$.
\end{enumerate}
\end{proposition}
\begin{proof}
\begin{enumerate}\itemsep0mm 
\item[(a)] Since the independence number of a set-graph is, $\alpha(G_{A^{(n)}}) = n$, the vertex set $X = \{v_{1,1}, v_{1,2}, v_{1,3},\dots, v_{1,n}\}$ is a maximum independent set. Also, the induced subgraph, $\langle V(G_{A^{(n)}}) - X\rangle$ is complete hence, the result.

\item[(b)] By colouring both end-vertices, $c(v_1) = c(v_n) = c_1$ the result is immediate.

\item[(c)] A consequence of (b).

\item[(d)] By merging vertices vertices $v_1,v_n$ of a path $P_n$, $n \geq 4$ and even, we obtain a cycle $C_n$, $n \geq 3$ and odd. Therefore the result for an odd cycle from (b). Similarly the result follows for an even cycle.

\item[(e)] Because the $n$ pendant vertices of a sunlet graph is a maximum independent set all may be coloured $c_1$. Then the result follows from (d).

\item[(f)] The result follows through similar reanoning found in (d) and (e).

\item[(g)] The result follows through similar reanoning found in (d) and (e).\\

\item[(h)]  Because the $\sum\limits_{i=1}^{n}t_i$ pendant vertices of a thorn(y) complete graph is a maximum independent set all may be coloured $c_1$. Then the result follows from the fact that, $\chi^{i-max}(K_n) = \chi(K_n)$.
\end{enumerate}
\end{proof}

\noi A \textit{complete thorn graph} $G^\star_c$ corresponding to a graph $G$ of order $n$ is the graph obtained by attaching $t_i \geq 1$ pendant vertices to each vertex $v_i\in V(G)$.

\begin{corollary}\label{Cor-3.2}
A complete thorn graph $G^\star_c$ has $\chi^{i-max}(G^\star_c)=\chi^{i-max}(G) +1$.
\end{corollary}
\begin{proof}
Since $\chi^{i-max}(G) \geq \chi(G)$ for any graph $G$ and a colouring in accordance with the parameter $\chi^{i-max}(G)$ always exists, the result upon constructing $G^\star_c$ follows immediately.
\end{proof}

It is proposed that a $\chi^-$-colouring hence, in accordance with a minimum proper colouring and in accordance with the rainbow neighbourhood convention, can be described as follows.

\begin{desc}\label{Col-3.2}{\rm  
Let $G$ have the distinct maximum independent sets, $X_i$, $1\leq i \leq k$. Consider the graph, $G' = G-X_i$ for some $X_i$ such that $\alpha(G') = max\{\alpha(G-X_j):1 \leq j \leq k\}$. Let us call the aforesaid procedure the first iteration and let $X_1 = X_i$. Colour all the vertices in $X_1$ the colour $c_1$. Repeat the procedure iteratively with $G'$ to obtain $G''$, then $G'''$, and proceeding like this, finally the empty graph. Colour the vertices in the maximum independent set obtained from the $j^{th}$-iteration the colour, $c_j$.
}\end{desc}

\begin{theorem}\label{Thm-3.3}
If a graph $G$ is coloured in accordance with Colouring description \ref{Col-3.2}, it results in a colouring of $G$ in accordance with the rainbow neighbourhood convention and in accordance with a minimum proper colouring of $G$.
\end{theorem}
\begin{proof}
Consider any graph $G$ that is coloured with $|\mathcal{C}| = \ell \geq 2$ colours in accordance with Colouring description \ref{Col-3.2}. It implies that $\theta(c_i) \geq 1$, $1\leq i \leq \ell$. Hence, a minimum parameter colouring set is complied with. Also, since each set $X_i$ is an independent set of vertices the colouring is a proper colouring. Clearly the colouring is in accordance with the $max$-requirement of the rainbow neighbourhood convention in that, after a maximum number of vertices were allocated say, colour $c_j$, a maximum number of vertices were subsequently assigned with the colour $c_{j+1}$.

Assume that it is not a minimum proper colouring of $G$. It then implies that a different proper colouring exists in which at least the vertices with colour $c_\ell$ can be coloured in some way with the colour set, $\{c_i:1\leq i \leq \ell -1\}$ to obtain a minimum proper colouring of $G$. Assume that this new minimum proper colouring results in at least, either $\theta(c_i)$ and $\theta(c_{i+1})$ to increase or only that $\theta(c_i)$ increases. Note that only an increase is possible because in some way, $\theta(c_\ell)$ vertices must become elements of lesser independent vertex sets. This implies a contradiction in at least the $i^{th}$-iteration of applying the Colouring description \ref{Col-3.2}. Immediate induction through iterations $1,2,3,\ldots, \ell$  shows that a colouring in accordance with Colouring description \ref{Col-3.2} results in a colouring of $G$ in accordance with the rainbow neighbourhood convention and in accordance with a minimum proper colouring of $G$. Therefore, $\chi(G) = \ell$.
\end{proof}

We state the next corollary without proof as it is an immediate consequence of Theorem 3.3.

\begin{corollary}\label{Cor-3.4}
For any graph $G$ we have: $\chi(G) \leq \chi^{i-max}(G)$.
\end{corollary}

\begin{theorem}\label{Thm-3.5}
For any graph of order $n \geq 2$ and for which, $\alpha(G) = \chi(G)$, we have: $\chi^{i-max}(G) = \chi(G) + 1$.
\end{theorem}
\begin{proof}
The result is an immediate consequence of Colouring description 3.1.
\end{proof}
\section{Conclusion}

In this paper, some interesting results for the rainbow neighbourhood number of set-graphs were presented. Set graphs were shown to be perfect. Conjecture \ref{Conj} remains open to be proved or disproved.

Furthermore, the paper clarified the intuitive colouring dilemma in respect of the rainbow neighbourhood convention. Colouring description \ref{Col-3.1} immediate lends the opportunity to research a new parameter denoted, $r^{i-max}_\chi(G)$. The aforesaid parameter is the number of vertices in a graph $G$ which yield rainbow neighbourhoods in $G$ if a minimum proper colouring in accordance with Colouring description \ref{Col-3.1} is applied.

An immediate consequence of Proposition \ref{Prop-3.1}, Corollary \ref{Cor-3.4} and Theorem \ref{Thm-3.5} is that for a sufficiently large $k \geq 1$, $k \in \N_0$, the bounds, $\chi(G) \leq \chi^{i-max}(G) \leq \chi(G) + k$ holds for any graph $G$. Also since $G$ is finite, such finite $k$ exists.We note that for all classes of graphs in Proposition \ref{Prop-3.1} the minimum $k$ is, $0$ or $1$. This observation motivates us in defining a new graph parameter as follows:

\begin{definition}{\rm 
The \textit{i-max number} of a graph $G$, denoted by $\alpha^{i-max}(G)$, which is defined to be $\alpha^{i-max}(G)=\min\{\ell:\chi^{i-max}(G)\leq \chi(G)+ k,\ k \in \N_0\}$. 
}\end{definition}

Finding an efficient algorithm to determine $\alpha^{i-max}(G)$ for classes of graphs remains open. Finding such algorithm will be a worthy challenge.

\end{document}